\documentclass[12pt]{article}
\usepackage[amsmath]{e-jc}
\usepackage{mathtools} 
\usepackage{tikz}
\usetikzlibrary{shapes, arrows, positioning}
\usetikzlibrary{positioning, quotes}
\makeatletter
\@ifundefined{copy@right}{
  \def\copy@right{\textcopyright}%
}{}
\makeatother

\usepackage{graphicx}

\dateline{Feb 15, 2025}{TBD}{TBD}


\title{Impartial Avoidance Games\\ on Convex Geometries}

\author{Seomgeun Shim\authornote{1}}

\authortext{1}{Sophomore at Hammond School, Columbia, South Carolina
   (\email{shimmiguel00@gmail.com}.)}

\begin{document}

\maketitle

\begin{abstract}

  We analyze a two-player game in which players take turns avoiding the selection of certain points within a convex geometry. The objective is to prevent the convex closure of all chosen points from encompassing a predefined set. The first player forced into a move that results in the inclusion of this set loses the game. We redevelop a theoretical framework for these avoidance games and determine their nim numbers, including cases involving vertex geometries of trees, edge geometries of trees, and scenarios where the predefined set consists of extreme points.
\end{abstract}

\section{Introduction}

A convex geometry is an abstract generalization of the notion of convexity in Euclidean space. In this paper, we study an avoidance game that is a type of misere game where two players take turns selecting previously unchosen points from a convex geometry. Following the normal play convention, our version of the avoidance game does not allow for the convex closure of chosen points to contain the winning set. The winning set is predefined before starting the game. The winner of the game is the last player able to move.

Originally, this game was introduced by Anderson and Harary \cite{Anderson} as a game using groups and further developed in \cite{Ernst} and \cite{Benesh}. In this paper, we study a variation of this game played in convex geometries, like \cite{McCoy}. Our research builds upon prior studies of Achievement games, a type of convex geometry game in which two players aim to enclose a winning set rather than avoid it. Throughout this research paper, we use the structural diagram and equivalences developed in \cite{Ernst} and \cite{McCoy} to compute the nim numbers of different types of convex geometries like edge geometries, vertex geometries, and affine geometries.  The key new result of this paper is the determination of the spectrum of nim numbers for avoidance games and the sum of avoidance games. While \cite{McCoy} examined achievement games in convex geometries, this work provides the first characterization of avoidance games in tree structures and when winning sets consist of extreme points, computing their nim numbers through a modified structural induction approach (due to the nature of misere games).

The paper is organized as follows: In Section 2, we review the basic terminologies of impartial games and convex geometries and establish the notation used throughout the paper. In Section 3, we introduce structural equivalence and structure diagram for avoidance games and prove it. In Section 4, we determine the spectrum of nim numbers of games where the winning set consists of extreme points. In Section 5, we determine the spectrum of nim numbers of games played on vertex geometries of trees and edge geometries of trees, which are similar in nature. In Section 6, we make our concluding remarks and propose a conjecture for further study.

\section{Preliminaries}

First, we recall the basic terminologies of impartial games, which are given a thorough treatment in \cite{Siegel}. 

An \emph{impartial game} is a finite set of positions with a starting position and a collection of options of position $P$. In each move of the game, two players take turns selecting a move from \(\text{Opt}(P)\), the set of valid options for the current position \( P \). If a player is faced with an empty option set, they cannot make a move and lose the game. In the context of this paper, an empty option occurs when every move results in the convex closure of selected points including the winning set. Since all games must conclude in a finite number of turns, infinite sequences of play are not allowed. 

Each position in the game falls into one of two categories:  
An \emph{N-position} (Next-player win) means the player whose turn it is can force a win.  The other is a \emph{P-position} (Previous-player win) means the player who just moved has secured a winning strategy, leaving their opponent with no path to victory.

The \emph{minimum excludant} $\text{mex}(A)$ of a set $A$ of nonnegative integers is the smallest nonnegative integer not in the set $A$. The \emph{nim number} of a position $P$ of the game is defined recursively by 
\[
\text{nim}(P) := \text{mex}(\text{nim}(\text{Opt(P)}))
\]
Since the minimum excludant of the empty set is 0, the terminal positions of a game have a nim-number of 0. The nim-number of a game corresponds to the nim-number of its starting position. This value determines the game's outcome, as a position \( P \) is a P-position if and only if \( \text{nim}(P) = 0 \).

\begin{example}
    Suppose that we are given a set $A = \{0,1,4,5 \}$ and $B = \{ 1,2,3\}$. $\text{mex}(A) = 2$ because 2 is the least nonnegative integer that is not in $A$ and $\text{mex}(B) = 0$ as 0 is the least nonnegative integer that is not in $B$. 
\end{example}

The \emph{sum} of two games, \( P \) and \( R \), is denoted as \( P + R \) and is defined as the game whose set of options is 
\[
\text{Opt(P+R)} := \{Q+R \mid Q \in \text{Opt}(P)\} \cup \{P+S \mid S \in \text{Opt}(R)\}
\]
In other words, in each turn, a player can make a move in either game $P$ or $R$. It is well-known that the nim number of the sum of these two games can be determined by 
\[
\text{nim}(P+R) = \text{nim}(P) \oplus \text{nim}(R)
\]
where $\oplus$ is a bitwise XOR operator. 
The one-pile Nim game with \( n \) stones is represented by the nimber \( *n \). The set of possible moves from \( *n \) is given by:  

\[
\text{Opt}(*n) = \{ *0, *1, \dots, *(n - 1) \}
\]
A key result highlights the importance of nim numbers:  
\begin{theorem}
    (Sprague–Grundy): For any impartial game \( P \), we have  
\[
P = *\text{nim}(P)
\]
\end{theorem}

\begin{definition}
    For convention, we define $\text{pty}(X) := |X| \pmod 2$.
\end{definition}
For the remainder of this section, we recall the basics terminologies of the convex geometries from \cite{Ahrens} and \cite{McCoy}. 

\begin{definition}
A \emph{convex geometry} is a pair $(S, \mathcal{K})$, where $S$ is a finite set and $\mathcal{K}$ is a family of subsets of $S$ satisfying the following properties:
\begin{enumerate}
    \item $S \in \mathcal{K}$;
    \item $K, L \in \mathcal{K}$ implies $K \cap L \in \mathcal{K}$;
    \item $S \ne K \in \mathcal{K}$ implies that there exists some  $a \in S \setminus K$ such that $K \cup \{a\} \in \mathcal{K}$.
\end{enumerate}
The sets $K \in \mathcal{K}$ are called \emph{convex}.
\end{definition}

\begin{definition}
\noindent A convex geometry induces a convex closure operator \( \tau: 2^S \to 2^S \) defined by:  
\[
\tau(A) := \bigcap \{K \in \mathcal{K} \mid A \subseteq K\}.
\]  
A point \( a \) in a subset \( A \) of a convex geometry is called an \emph{extreme point} of \( A \) if \( a \notin \tau(A \setminus \{a\}) \). The set of all extreme points of \( A \) is denoted by \( \text{Ex}(A) \).
\end{definition}

\begin{example}
    A valid convex geometry on $S = \{1,2,3,4\}$ and $\mathcal{K} = \{\{2,3,4\}, \{1,2,3,4\}\}$. Another valid convex geometry is when $S=\{1,2,3,4\}$  and $\mathcal{K} = \{\{2,3\}, \{2,3,4\}, \{1,2,3\}, \{1,2,3,4\}\}$.
\end{example}
\begin{example}
    Suppose we take $S$ as a finite subset of $\mathbb{R}^2$. Then, the convex closure operator of set $S$ is the convex hull of $S$.
\end{example}

Since we have covered the necessary terminology, we can move on to the detailed description of avoidance game $\text{DNG}(S,W)$ played on convex geometry $(S,K)$, with a fixed, nonempty winning subset of $W \in S$. In the $n$th turn, a player can chose the point $p_n$ if and only if the convex closure of jointly selected points $P = \{p_1,p_2, \dots, p_{n-1},p_n\}$ is not a superset of $W$. In other words, if
\[
W \not\subset \tau(P) 
\]
the player can choose the point $p_n$. The last player to move wins the game. We call $P$ the \textit{current position} of the game. \newline

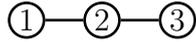
\begin{figure}[htp]
    \centering
    \begin{tikzpicture}[
        scale=2,line width=1pt,
        node distance=0.3cm and 0.5cm, 
        every node/.style={draw, circle, align=center, minimum size=6pt, inner sep=1.2pt},
        blue/.style = {fill=blue!20},
        every edge quotes/.style = {auto, font=\footnotesize, sloped}
    ]

    \node (1) at (0,0) {1};
    \node (2) [right=of 1] {2};
    \node (3) [right=of 2] {3};

    \path (1) edge (2);
    \path (3) edge (2);
    \end{tikzpicture}
    \caption{Affine convex geometry in Example 8}
    \label{fig:convex-geometry-graph}
\end{figure}

\begin{example}
    Suppose we play the avoidance game on affine convex geometry in $\mathbb{R}$ with $S = {1,2,3}$ and $W = {2}$, showing in Figure~\ref{fig:convex-geometry-graph}. Here, the convex sets are The first player can choose either $1$ or $3$, which are the same by symmetry. Therefore, we can assume that they choose $1$. On the next turn, the second player does not have any valid moves as $W = \{2\} \in \tau(\{1,2\})$ and $W \in \tau(\{1,3\})$. Thus, the first player always wins.
\end{example}
\section{Structure Theory}

Structural equivalence identifies equivalent game positions, significantly reducing the computation of nim numbers by allowing the use of a smaller quotient of the game digraph. Structure theory from \cite{Ernst} can be used to arrive at the main results in sections 4 and 5. We will prove all of the propositions but void the proof of the main theorem of this section as it is included in \cite{Ernst}.

Consider the avoidance game \( GEN(S, W) \), where \( S \) represents a set and \( W \) is a winning subset. A subset \( M \subseteq S \) is referred to as a generating set if it satisfies \( W \subseteq \tau(M) \). If this condition is not met, \( M \) is termed a non-generating set. Furthermore, if \( M \) is non-generating but every larger subset \( N \) with \( M \subset N \subseteq S \) is generating, then \( M \) is classified as maximally non-generating. The collection of all maximally non-generating subsets is denoted by \( \mathcal{M} \).  

\begin{proposition}
    Every maximally non-generating set \( M \) is convex.  
\end{proposition}

\begin{proof}
    Note that a set $M$ is convex if and only if $\tau(M) = M$. Assume the contradiction and suppose that $\tau(M) \neq M$. Then, we can choose $x \in \tau(M)$ not in $M$, so the set is not maximally non-generating, contradiction. Therefore, $\tau(M) = M$, so $M$ has to be convex if it is maximally non-generating.
\end{proof}

\begin{definition}
    Define the set of intersection subsets as  
    \[
    \mathcal{I} := \{\bigcap_{\mathcal{N} \subseteq \mathcal{M}} \mathcal{N}\}
    \]  
    which consists of all elements common to every maximally non-generating subset. The smallest such subset, known as the Frattini subset \( \Phi \), is the intersection of all maximally non-generating subsets. We ignore $\bigcap(\emptyset)$ from getting added to $\mathcal{I}$, unlike in \cite{McCoy} because this is an avoidance game, so the terminal position for which $W \subseteq \tau(P)$ does not exist. 
\end{definition}

\begin{definition}
    For some position $P$ in the game $\text{DNG}(S,W)$, define 
    \[
    \lceil P \rceil := \bigcap\{M\in\mathcal{M} \mid P\subseteq M\}
    \]
    Two game positions are structure equivalent if $\lceil P \rceil = \lceil R \rceil$. Because $\lceil I \rceil = I$ by the definition of $I$, we can let $X_I$ be the structure class under this equivalence relation.
\end{definition}

Two following theorems proved in \cite{Ernst} lets us use structural diagrams.
\begin{theorem}
    If $P, Q \in X_I$ and $\text{pty}(P) = \text{pty}(Q)$, then $\text{nim}(P) = \text{nim}(Q)$. 
\end{theorem}
\begin{definition}
    We define the \emph{type} of class $X_I$ as 
    \[
    \text{type}(X_I) = (\text{pty}(I), \text{nim}_0(X_I), \text{nim}_1(X_I))
    \]
    where $\text{nim}_0(X_I)$ is the nim number of even parity positions in $X_I$ and $\text{nim}_1(X_I)$ is the nim number of odd parity position in $X_I$. 
\end{definition}
\begin{definition}
    We define $\text{Opt}(X_I) = \{X_J \mid J = \text{Opt}(P) \text{ where } P\in X_I\}$
\end{definition}
\begin{theorem}
    We can compute the type of class $X_I$ by recursion relation
    \[
    \text{nim}_{\text{pty}(I)}(X_I) = \text{mex}(\text{nim}_{1-\text{pty}(I)}(\text{Opt}(X_I)))
    \]
    and 
    \[
    \text{nim}_{1-\text{pty}(I)}(X_I) = \text{mex}(\text{nim}_{\text{pty}(I)}(\text{Opt}(X_I))\cup{\text{nim}_{\text{pty}(I)}(X_I)})
    \]
\end{theorem}

\begin{figure}
    \centering
    \includegraphics[width=0.3\linewidth]{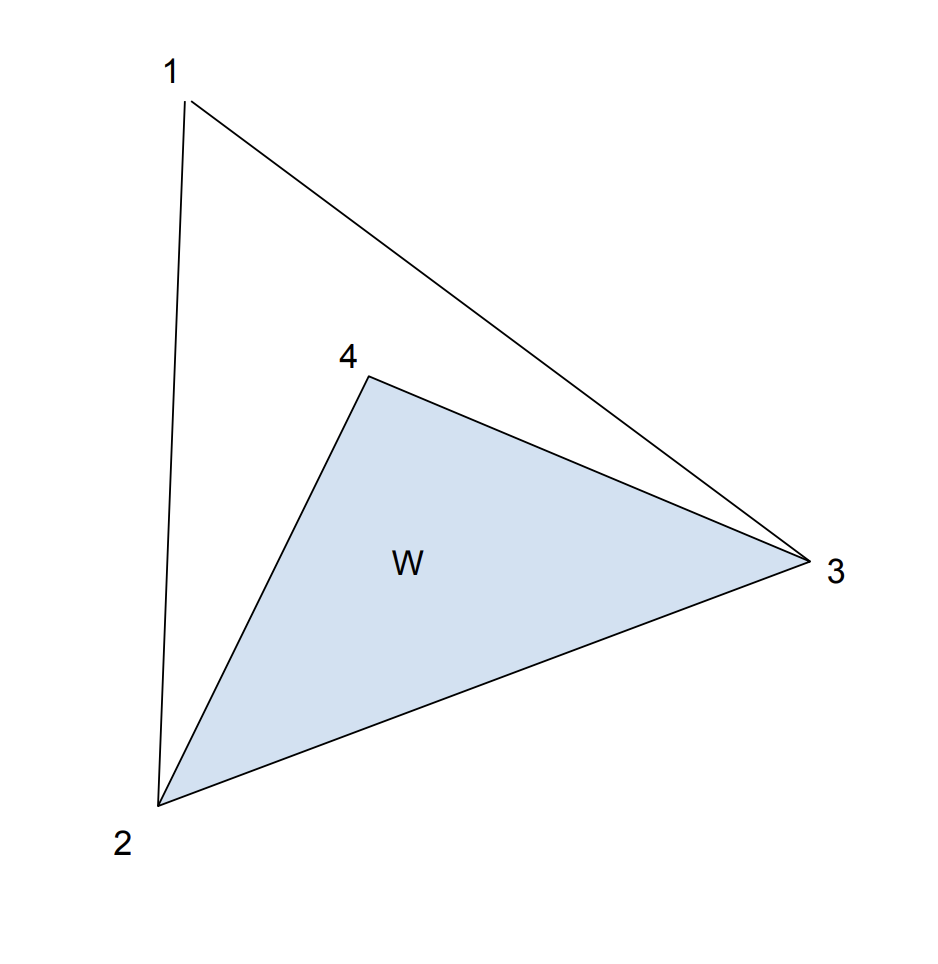}
    \caption{Diagram for Example 15}

\end{figure}
\begin{figure}[htp]\centering
    \begin{tikzpicture}[
        scale=3,line width=1pt,
        node distance=0.6cm and 1.0cm, 
        triangle/.style={draw, shape=regular polygon, regular polygon sides=3, minimum size=1cm, inner sep=1pt, align=center},
        upside down/.style={shape border rotate=180},
        every path/.style={->, thick}
    ]

    \node[triangle] (1) at (0,0) {0,1};
    \node[triangle] (2) at (0,1) {3,2};
    \node[triangle, upside down] (3) at (1,0) {1,0};
    \node[triangle] (4) at (1,1) {1,0};
    \node[triangle, upside down] (5) at (2,0) {1,0};
    \node[triangle] (6) at (2,1) {1,0};
    \node[triangle] (7) at (1,2) {1,0};
    \path (7) edge (2);
    \path (7) edge (4);
    \path (7) edge (6);
    \path (2) edge (1);
    \path (2) edge (3);
    \path (4) edge (3);
    \path (4) edge (5);
    \path (6) edge (1);
    \path (6) edge (5);
    \end{tikzpicture}
\end{figure}
\begin{example}
    Figure above shows a point set $S = {1,2,3,4}$ and $W = {2,3,4}$ in $\mathbb{R}^2$. The maximally non-generating sets are $\mathcal{M} = \{\{2,3\}, \{1,2,4\}, \{1,3,4\}\}$, so \newline $\mathcal{I} = \{\{2,3\}, \{1,2,4\}, \{1,3,4\}, \{2\}, \{1,4\}, \{3\},\emptyset\}$ and $\Phi = \emptyset$. Therefore, the we get the structure diagram shown below, so the nim number of the game is 1.
\end{example}
Using these we can create structural diagrams. In this paper, we use the conventions of structure diagrams used in \cite{McCoy}.

\begin{lemma}
    If structure class $X_I$ is terminal, then $\text{type}(X_I)$ must be either $(0,0,1)$ or $(1,1,0)$.
\end{lemma}

\begin{proof}
    Suppose that $\text{pty}(X_I) = 0$. Then, the even positions are immediately losing for the first player as $I$ is terminal. Therefore, $\text{type}(X_I) = (0,0,1)$. When $\text{pty}(X_I) = 1$, the odd positions are immediately losing so the type is $(1,1,0)$. 
\end{proof}
\section{Winning Subsets that Only Includes Extreme Points}
In this section, we characterize the nim numbers of $\text{DNG}(S,W)$ where $\tau(W) \subseteq \text{Ex}(S)$. 

\begin{proposition}
    The nim numbers of $\text{DNG}(S,W), \text{DNG}(S,\text{Ex}(W)), \text{DNG}(S,\tau(W))$ are the same.
\end{proposition}
\begin{proof}
    Suppose we have a position $P \subseteq S$. Notice that $\tau(W) = \tau(\text{Ex}(W)) = \tau(\tau(W))$. Therefore, $\text{Ex}(W) \not\subseteq \tau(P) \implies \tau(W) = \tau(\text{Ex}(W)) \not\subseteq \tau(\tau(P)) = \tau(P)$. We can similarly prove that one of the conditions imply the other, so we are done.
\end{proof}

Proposition 7 lets us assume that $W \subseteq \text{Ex}(S)$. We first look at the special case when $W = \text{Ex}(S)$. Because of the definition of an extreme point, $\mathcal{M} = \{S \backslash \{v\} \mid v \in \text{Ex}(S)\}$ and therefore $\mathcal{I} = \{S \backslash V \mid V \subseteq S\}$.

\begin{proposition}
    If $W = \text{Ex}(S)$, then 
    \[
\text{nim}(\text{DNG}(S, W)) =  
   \begin{dcases} 
      1, & \text{pty}(S) = 0 \\
      0, & \text{pty}(S) = 1 \\
   \end{dcases}
    \]
\end{proposition}
\begin{figure}[htp]\centering
    \begin{tikzpicture}[
        scale=3,line width=1pt,
        node distance=0.6cm and 1.0cm, 
        triangle/.style={draw, shape=regular polygon, regular polygon sides=3, minimum size=1cm, inner sep=1pt, align=center},
        upside down/.style={shape border rotate=180},
        every path/.style={->, thick}
    ]

    \node[triangle, upside down] (11) at (0,0) {1,0};
    \node[triangle] (21) [above =of 11] {1,0};
    \node[triangle, upside down] (31) [above =of 21] {1,0};
    \node[triangle] (41) [above =of 31] {1,0};
    \node[triangle] (12) at (2,0) {0,1};
    \node[triangle, upside down] (22) [above =of 12] {0,1};
    \node[triangle] (32) [above =of 22] {0,1};
    \node[triangle, upside down] (42) [above =of 32] {0,1};

    \path (22) edge (12);
    \path (32) edge (22);
    \path (42) edge (32);

    \path (21) edge (11);
    \path (31) edge (21);
    \path (41) edge (31);
    \end{tikzpicture}
\end{figure}
\begin{proof}
    We proceed by induction on the orbit quotient structure diagram. When $\text{pty}(S) = 0$, maximum non-generating sets are of even parity as they are one point removed from $S$. By induction, $\text{type}(X_I) = (0,1,0)$ or $(1,1,0)$ Hence, the orbit quotient diagram gives us $\text{nim}(\text{DNG}(S,W)) = 1$ by induction, as shown in the figure. When $\text{pty}(S) = 1$, the maximum non-generating sets are of odd parity with similar reasoning. Thus, the orbit quotient diagram gives $\text{type}(X_I) = (0,0,1)$ or $(1,0,1)$ so $\text{nim}(\text{DNG}(S,W)) = 0$ by induction, as shown in the figure.
\end{proof}

We can generalize this result to winning sets where $W \subseteq \text{Ex}(S)$. Next few results are from \cite{McCoy}:

\begin{proposition}
    If $W \subseteq \text{Ex}(S)$, then the set $\mathcal{M}$ of maximally non-generating subsets of $\text{DNG}(S,W)$ is $\{M_v \mid v \in W\}$, where $M_v := S \backslash \{v\}$. 
\end{proposition}
\begin{proof}

    We first show that for each \( v \in W \), \( M_v \) is a maximally non-generating set. Since \( v \) is an extreme point of \( S \), it follows that \( v \notin \tau(M_v) \), implying that \( M_v \) is a non-generating set. Moreover, if \( M_v \subset N \subseteq S \), then necessarily \( N = S \), which is a generating set. This confirms that \( M_v \) is maximally non-generating.  
    
    Next, we establish that every maximally non-generating set must be of the form \( M_v \) for some \( v \in W \). Let \( M \) be a maximally non-generating set. Since \( M \) is non-generating, there must exist some \( v \in W \setminus M \). Clearly, we have \( M \subseteq M_v \), meaning \( M_v \) is a non-generating superset of \( M \). Given that \( M \) is maximally non-generating, it must be that \( M = M_v \).  
\end{proof}

\begin{corollary}
    If $W \subseteq \text{Ex}(S)$, then the set of intersection subsets of $\text{DNG}(S,W)$ is $\mathcal{I} = \{S \backslash V \mid V \subseteq W \}$. \newline
\end{corollary} 

Now, we can get to our main result of this section: 

\begin{definition}
    For $I \in \mathcal{I}$, we define $\delta(I) := |S\backslash I|$. In other words, it is the number of points missing from $I$.
\end{definition}

\begin{proposition}
    If $W \subseteq \text{Ex}(S)$, 
    \[
\text{nim}(\text{DNG}(S, W)) =  
   \begin{dcases} 
      1, & \text{pty}(S) = 0 \\
      0, & \text{pty}(S) = 1 \\
   \end{dcases}
    \]
\end{proposition}
\begin{figure}[htp]\centering
    \begin{tikzpicture}[
        scale=3,line width=1pt,
        node distance=0.6cm and 1.0cm, 
        triangle/.style={draw, shape=regular polygon, regular polygon sides=3, minimum size=1cm, inner sep=1pt, align=center},
        upside down/.style={shape border rotate=180},
        every path/.style={->, thick}
    ]

    \node[triangle, upside down] (11) at (0,0) {1,0};
    \node[triangle] (21) [above =of 11] {1,0};
    \node[triangle, upside down] (31) [above =of 21] {1,0};
    \node[triangle] (41) [above =of 31] {1,0};
    \node[triangle] (12) at (2,0) {0,1};
    \node[triangle, upside down] (22) [above =of 12] {0,1};
    \node[triangle] (32) [above =of 22] {0,1};
    \node[triangle, upside down] (42) [above =of 32] {0,1};

    \path (22) edge (12);
    \path (32) edge (22);
    \path (42) edge (32);

    \path (21) edge (11);
    \path (31) edge (21);
    \path (41) edge (31);
    \end{tikzpicture}
\end{figure}
\begin{proof}
    We are going to use type calculus and structural induction to prove this result. First, we consider when $\text{pty}(S) = 0$. This can be visualized in the right diagram above. We are going to prove that in this case, 
    \[
    \text{type}(X_I) = \begin{dcases}
        (0,0,1), & \delta(I) \equiv1 \pmod2 \\
        (1,0,1), & \delta(I) \equiv 0 \pmod2\\
    \end{dcases}
    \]
    First, note that we start with the maximally non-generating set which means $\delta(I) = 1$ and has type calculus $(1,0,1)$ since $\text{pty}(I) = 1-\text{pty}(S) = 1$. This will serve as our base case. We assume that the result holds when $\delta(I) = 2k+1$ for a nonnegative integer $k$. When $\delta(I) = 2k+2$, we get 
    \begin{equation}
        \begin{split}
            \text{type}(X_I)  & = (0,\text{nim}_0(X_I), \text{nim}_1(X_I)) \\
             & = (0,\text{mex}(\text{nim}_1(\text{Opt}(I))), \text{mex}(\text{nim}_1(\text{Opt}(I)))\cup {\text{nim}_0(X_I)})\\
             & = (0, \text{mex}(0), \text{mex}(1)) \\
             & = (0,1,0)
        \end{split}
    \end{equation}

    \noindent Next, when $\delta(I) = 2k+3$, we get that
    \begin{equation}
        \begin{split}
            \text{type}(X_I)  & = (1,\text{nim}_0(X_I), \text{nim}_1(X_I)) \\
             & = (1,\text{mex}(\text{nim}_1(\text{Opt}(I)))\cup {\text{nim}_1(X_I)}), \text{mex}(\text{nim}_0(\text{Opt}(I)))\\
             & = (1, \text{mex}(0), \text{mex}(1)) \\
             & = (1,1,0)
        \end{split}
    \end{equation}
    Therefore, our induction is complete. 

    Now consider the cases when $\text{pty}(S) = 1$. This can be visualized as left diagram above. We can similarly prove that 
    \[
    \text{type}(X_I) = \begin{dcases}
        (1,1,0), & \delta(I) \equiv1 \pmod2 \\
        (0,1,0), & \delta(I) \equiv 0 \pmod2\\
    \end{dcases}
    \]
    with induction, which is essentially the same as the argument above with the parities change. 

\end{proof}
\begin{corollary}
    The spectrum of nim numbers when $\tau(W) \subseteq \text{Ex}(S)$ is $\{0,1\}$.
\end{corollary}

\begin{proposition}
    The spectrum of nim numbers for the sum of games where $\tau(W) \subseteq \text{EX}(S)$ is $\{0,1\}$.
\end{proposition}
\begin{proof}
    We use the bitwise XOR operator to compute the nim number of the sum of the game. $\{0,1\}$ are obviously obtainable as $0 \oplus 0 = 0$ and $0 \oplus 1 = 1$.  The set $\{0,1\}$ is closed under bitwise XOR as $0 \oplus 0 = 1$, $1 \oplus 0 = 1$, $1 \oplus 1 = 0$, and $0 \oplus 1 = 1$.  Therefore, we are done.
\end{proof}

\section{Vertex and Edge Geometry on Trees}
The concept of convex geometries can be naturally extended to a tree graph \( T \) with a vertex set \( S \). Specifically, the vertex sets of connected subgraphs of \( T \) constitute a convex geometry on \( S \). We refer to this structure as the \emph{vertex geometry} of \( T \). In this section, we study the avoidance game $\text{DNG}(S,W)$ on the vertex geometry and edge geometries of trees. 

\begin{figure}[htp]\centering
    \begin{tikzpicture}[
        scale=2,line width=1pt,
        node distance=0.3cm and 0.5cm, 
        every node/.style={draw, circle, align=center, minimum size=6pt, inner sep=1.2pt},
        every edge quotes/.style = {auto, font=\footnotesize, sloped}
    ]

    \node (1) at (0,0) {1};
    \node (2) [left =of 1] {2};
    \node (3) [left =of 2] {3};
    \node (4) [above left= of 3] {4};
    \node (5) [below left= of 3] {5};
    \node (6) [right =of 1] {6};
    \node (7) [above right =of 6] {7};
    \node (8) [above right =of 7] {8};
    \node (9) [below right =of 7] {9};
    \node (10) [below right =of 6] {10};
    \path (1) edge (2);
    \path (3) edge (2);
    \path (4) edge (3);
    \path (5) edge (3);
    \path (1) edge (6);
    \path (6) edge (7);
    \path (7) edge (8);
    \path (7) edge (9);
    \path (10) edge (6);

    \end{tikzpicture}
\end{figure}

\begin{example}
    An example of a convex set for the vertex geometry shown above are sets $\{1,2,3\}$, $ \{6,7,9,10\}$, and $\{1,2,3,4,6,7\}$ because they are connected subgraphs of the tree. However, sets such as $\{2,3,6,7,9\}$ and $\{4,10\}$ are not convex because the components are not connected by the edges of the tree.
\end{example}

We first study the nim numbers of a vertex geometry. We bring in some of the definitions and results introduced in \cite{McCoy} and adapt this into the $\text{DNG}$ game. 

\begin{definition}
    For a vertex $v$ in the tree graph $T$, let $N(v)$ denote the set of vertices adjacent to $v$.
\end{definition}

\begin{proposition}
     Let \( (S, \mathcal{K}) \) represent the vertex geometry of a tree \( T \), and suppose \( W \subseteq S \). Then, a vertex \( w \) belongs to \( Ex(W) \) if and only if the set \( W \setminus \{w\} \) is a single connected component of \( T\backslash w \)   
\end{proposition}

\begin{proof}
    Consider each \( v \in N(w) \), where \( V_v \) denotes the set of vertices forming the connected component of \( T \setminus w \) that includes \( v \). The mapping \( v \mapsto V_v \) establishes a one-to-one correspondence between the neighbors of \( w \) and the vertex sets of the connected components of \( T \setminus w \).  
    First, assume that two vertices \( w_1, w_2 \in W \setminus \{w\} \) reside in separate connected components of \( T \setminus w \). In this case, any connected subgraph of \( T \) that contains both \( w_1 \) and \( w_2 \) must also include \( w \). Consequently, \( w \) belongs to \( \tau(W \setminus \{w\}) \), implying \( w \notin Ex(W) \).  
    
    Conversely, suppose \( W \setminus \{w\} \) is entirely contained within a single connected component \( V_v \) of \( T \setminus w \). Since \( V_v \) is convex and encloses \( W \setminus \{w\} \), it follows that \( \tau(W \setminus \{w\}) \subseteq V_v \). Given that \( w \notin V_v \), we conclude that \( w \in Ex(W) \).
\end{proof}

There are two cases to consider to completely solve fro the nim numbers in the vertex geometry: when $|W| = 1$ and when $|W| \geq 2$. We first consider the case when $|W| =1$. 

Suppose that $W = \{w\} \subseteq S$. For each $v \in N(w)$, there exists a unique connected component, $M_v$ of $T\backslash w$ that contains $v$. The vertex set $M_v$ is a maximally non-generating set because the next move will have to include $W = \{w\}$. Therefore, $\mathcal{M} = \{M_v \mid v \in N(w)\}$ and all of the elements of $\mathcal{M}$ are disjoint by construction. 

\begin{definition}
    The \emph{signature} of a game is $(e,o)$ in which $e$ is the cardinality of the set $\{M_v \in \mathcal{M} \mid \text{pty}(M_v) = 0\}$ and $o$ is the cardinality of the set $\{M_v \in \mathcal{M} \mid \text{pty}(M_v) = 1\}$. In other words, $e$ is the number of even-parity sets $M_v$ and $o$ is the number of odd-parity sets $M_v$.
\end{definition}
\begin{figure}[htp]\centering
    \begin{tikzpicture}[
        scale=2,line width=1pt,
        node distance=0.3cm and 0.5cm, 
        every node/.style={draw, circle, align=center, minimum size=6pt, inner sep=1.2pt},
        blue/.style = {fill=blue!20},
        every edge quotes/.style = {auto, font=\footnotesize, sloped}
    ]

    \node[blue] (1) at (0,0) {1};
    \node (2) [left =of 1] {2};
    \node (3) [above left =of 2] {3};
    \node (4) [below left= of 2] {4};
    \node (5) [above = of 1] {5};
    \node (6) [above =of 5] {6};
    \node (7) [below right =of 1] {7};
    \node (8) [below right =of 7] {8};
    \node (9) [above right =of 1] {9};
    \path (1) edge (2);
    \path (3) edge (2);
    \path (4) edge (2);
    \path (5) edge (1);
    \path (5) edge (6);
    \path (1) edge (9);
    \path (1) edge (7);
    \path (7) edge (8);

    \end{tikzpicture}
\end{figure}

\begin{example}
    The diagram above shows a tree graph with vertex set $S = \{1,2,3,\dots, 9\}$. The winning set is $W= \{1\}$, as shaded with blue. $N(w) = N(1) = \{2,5,7,9\}$, so we have that $\mathcal{M} = \{M_2, M_5, M_7, M_9\}$ where $M_2 = \{2,3,4\}$, $M_5 \ \{5,6\}$, $M_7 = \{7,8\}$, and $M_9 = \{9\}$. The Frattini subset is $\Phi = \cap\mathcal{M} = \emptyset$. Using this, we can calculate a structure diagram as shown below, so $\text{nim}(\text{DNG}(S,W)) = 2$.
\end{example}
\begin{figure}[htp]\centering
    \begin{tikzpicture}[
        scale=3,line width=1pt,
        node distance=0.6cm and 1.0cm, 
        triangle/.style={draw, shape=regular polygon, regular polygon sides=3, minimum size=1cm, inner sep=1pt, align=center},
        upside down/.style={shape border rotate=180},
        every path/.style={->, thick}
    ]

    \node[triangle] (1) at (0,0) {3,2};
    \node[triangle] (2) at (-1.5,-0.5) {0,1};
    \node[triangle] (3) at (-0.5, -0.5) {0,1};
    \node[triangle,upside down] (4) at (0.5,-0.5) {1,0};
    \node[triangle,upside down] (5) at (1.5, -0.5) {1,0};
    \path (1) edge (2);
    \path (1) edge (3);
    \path (1) edge (4);
    \path (1) edge (5);

    \end{tikzpicture}
\end{figure}
\begin{proposition}
    Let $(S, \mathcal{K})$ be the vertex geometry of a tree $T$ and $W = \{w\} \subseteq S$. Given the signature of $(e,o)$, then 
    \[
    \text{nim}(\text{DNG}(S, W)) =  
    \begin{dcases} 
        0, & e=0, o=0\\
        0, & e \geq 1, o=0 \\
        1, & e=0, o \geq 1 \\
        3, & e\geq 1, o\geq 1
    \end{dcases}
    \]
\end{proposition}
\begin{proof}
    The diagram is shown below in the order of the cases in the proposition. When $e=0$ and $o=0$, then the starting position is the terminal position, which means that the nim number is 0. 

    When $e \geq 1$ and $o = 0$, we have to consider two cases: $e=1$ and $e >1$. When $e=1$, the Frattini subset is the maximally non-generating set itself, yielding nim number of $0$ as its type is $(0,0,1)$. When $e >1$, the Frattini subset is an empty set (as two maximally non-generating set cannot share an element) and $\text{Opt}(\Phi) = \mathcal{M}$, so 
    \begin{equation}
        \begin{split}
            \text{type}(X_\Phi)  & = (0,\text{nim}_0(X_\Phi), \text{nim}_1(X_\Phi)) \\
             & = (0,\text{mex}(\text{nim}_1(\text{Opt}(\Phi))), \text{mex}(\text{nim}_0(\text{Opt}(\Phi)))\cup {\text{nim}_0(\Phi)})\\
             & = (0, \text{mex}(1), \text{mex}(0)) \\
             & = (0,0,1)
        \end{split}
    \end{equation}

    When $o \geq 1$ and $e=1$, we also have two cases: $o=1$ and $o>1$. When $o=1$, the Frattini subset has to be maximally non-generating set itself because there is only one maximally non-generating set. This yields the nim number of 1 as its type is $(1,1,0)$ because $\text{pty}(X_\Phi) = \text{pty}(M) = 1$. When $o>1$, the Frattini subset is an empty set and $\text{Opt}(\Phi) = \mathcal{M}$, so  
    \begin{equation}
        \begin{split}
            \text{type}(X_\Phi)  & = (1,\text{nim}_0(X_\Phi), \text{nim}_1(X_\Phi)) \\
             & = (1,\text{mex}(\text{nim}_1(\text{Opt}(\Phi)))\cup {\text{nim}_1(\Phi)}), \text{mex}(\text{nim}_0(\text{Opt}(\Phi)))\\
             & = (1, \text{mex}(0), \text{mex}(1)) \\
             & = (1,1,0)
        \end{split}
    \end{equation}
    When $e\geq1$ and $o\geq 1$, the Frattini subset is empty as there are at least two maximally non-generating sets. Note that even cardinality maximally non-generating sets have type $(0,0,1)$ and odd cardinality maximally non-generating sets have type $(1,1,0)$. All of the maximally non-generating sets are options of $X_\Phi$, so
    \begin{equation}
        \begin{split}
            \text{type}(X_\Phi)  & = (0,\text{nim}_0(X_\Phi), \text{nim}_1(X_\Phi)) \\
             & = (0,\text{mex}(\text{nim}_1(\text{Opt}(\Phi))), \text{mex}(\text{nim}_0(\text{Opt}(\Phi)))\cup {\text{nim}_0(\Phi)})\\
             & = (0,\text{mex}(0,1,2), \text{mex}(0,1)) \\
             & = (0,3,2)
        \end{split}
    \end{equation}
    Therefore, the nim number is 3.
\end{proof}
\begin{figure}[htp]\centering
    \begin{tikzpicture}[
        scale=3,line width=1pt,
        node distance=0.6cm and 1.0cm, 
        triangle/.style={draw, shape=regular polygon, regular polygon sides=3, minimum size=1cm, inner sep=1pt, align=center},
        upside down/.style={shape border rotate=180},
        every path/.style={->, thick}
    ]

    \node[triangle, upside down] (1) at (0,0) {0,1};
    \node[triangle] (2) at (1,0) {1,0};
    \node[triangle] (3) [above =of 2] {1,0};
    \node[triangle, upside down] (4) at (2,0) {0,1};
    \node[triangle] (5) [above =of 4] {0,1};
    \node[triangle, upside down] (6) at (3,0) {0,1};
    \node[triangle] (7) [above right=of 6] {3,2};
    \node[triangle] (8) [below right =of 7] {1,0};
    \path (3) edge (2);
    \path (5) edge (4);
    \path (7) edge (6);
    \path (7) edge (8);
    \end{tikzpicture}
\end{figure}

Now we consider the case when $|W|\geq 2$. If \( w \) belongs to \( Ex(W) \), then removing \( w \) from \( T \) results in a single connected component that includes elements of \( W \) as we proved in Proposition 19. We denote the set of vertices in this component as \( M_w \). By construction, $M_w$ is a maximally non-generating set. We further define \( V_w := S \setminus M_w \), leading to a collection of disjoint sets \( \{V_w \mid w \in Ex(W)\} \).

There exists a one-to-one correspondence between the subsets of \( \text{Ex}(W) \) and the intersection subsets, expressed as \( A \mapsto M_A \), where:

\[
M_A := \bigcap \{M_w \mid w \in A\} = S \setminus \bigcup \{V_w \mid w \in A\}
\]

By definition, \( S = M_{\emptyset} \) and \( \emptyset = M_{\text{Ex}(W)} \). The deficiency of a structure class \( M_A \) is given by \( \delta(M_A) := |A| \). Additionally, the signature of a structure class \( X_{M_A} \) is represented as \( \sigma(X_{M_A}) := (e, o) \), where \( e \) denotes the number of elements \( a \in A \) for which \( pty(V_a) = 0 \), and \( o \) represents the number of elements \( a \in A \) for which \( pty(V_a) = 1 \). The deficiency follows the relation \( \delta(M_A) = e + o \). The overall signature of the game is given by \( \sigma(X_{\emptyset}) \).

\begin{proposition}
    Let $(S, \mathcal{K})$ be the vertex geometry of a tree $T$ and $W \subseteq S$ such that $|W| \geq 2$. Given the signature of the game is $(e,o)$, 
    \[
    \text{nim}(\text{DNG}(S,W)) = \begin{dcases}
        1, & o>e \\
        0, & o<e\\
        3, & o=e, o \equiv 1 \pmod2\\
        2, & o=e, o \equiv 0\pmod2\\
    \end{dcases}
    \]
    when $\text{pty}(S) = 0$ and 
    \[
    \text{nim}(\text{DNG}(S,W)) = \begin{dcases}
        0, & o>e \\
        1, & o<e\\
        3, & o=e, o \equiv 1 \pmod2\\
        2, & o=e, o \equiv 0\pmod2\\
    \end{dcases}
    \]
    when $\text{pty}(S) = 1$.
\end{proposition}

\begin{proof}
    Let \( X_J \) be an option of \( X_I \), meaning that we can move from \( X_I \) to \( X_J \) by selecting a new element. Since \( I = M_A \), adding an element \( v \) from \( J \setminus I \) results in \( I \cup \{v\} \in X_J \).  Now, there is exactly one unique element \( a \) in \( A \) such that \( v \) belongs to \( V_a \), and removing \( a \) from \( M_A \) gives us \( J = M_A \setminus \{a\} \).  
    
    This shows that each arrow in the structure diagram corresponds directly to a unique extreme point \( a \) of \( W \).  Since each move reduces \( |A| \) by exactly one, we see that the distance \( \delta(J) \) (the number of missing points) follows the pattern:  
    
    \[
    \delta(J) = |A| - 1 = \delta(I) - 1.
    \]
    In simpler terms, this means that moves always decrease the distance from \( X_S \) by exactly one step, and there are no direct connections between states that are the same distance away.

    Assume that \( I = M_A \) and that the signature of \( X_I \) is \( \sigma(X_I) = (e, o) \). If both \( e \) and \( o \) are at least 1, then there exist elements \( a, b \in A \) such that \( \text{pty}(V_a) = 0 \) and \( \text{pty}(V_b) = 1 \). Therefore, the signature of the options of \( X_I \) follows the pattern:  
    \[
    \sigma(\text{Opt}(X_I)) = \begin{dcases}
        \{e-1,0\}, & o=0 \\
        \{ 0,o-1\} & e=0\\
        \{e-1,o\}, \{e, o-1\}, & e,o \geq 1\\
    \end{dcases}
    \]

    Now, our result follows from type calculus and the structural induction. The details of this structural induction is shown in the diagram below for when $\text{pty}(S) = 0$. The only difference that occurs when $\text{pty}(S)= 1$ is that the signature $(o,e)$ is flipped; the signature that would be $(o,e)$ when $\text{pty}(S)=0$ is $(e,o)$ when $\text{pty}(S)=1$ due to the parity argument. Therefore, the result when $\text{pty}(S)=1$ follows. 
\end{proof}

\begin{figure}[htp]\centering
    \begin{tikzpicture}[
        scale=3,line width=1pt,
        node distance=0.6cm and 0.5cm, 
        every node/.style={draw, shape=regular polygon, regular polygon sides=3, minimum size=0.1cm, inner sep=0.1pt, align=center},
        down/.style={shape border rotate=180},
        every path/.style={->, thick}
    ]

    \node[down] (1) at (0,0) {1,0};
    \node (2) [above left=of 1] {1,0};
    \node[down] (3) [above right= of 1] {3,2};
    \node[down] (4) [above left=of 2] {1,0};
    \node (5) [above right= of 2] {1,0};
    \node[down] (6) [above right=of 3] {0,1};
    \node (7) [above left= of 4] {1,0};
    \node[down] (8) [above left=of 5] {1,0};
    \node (9) [above left= of 6] {2,3};
    \node[down] (10) [above right=of 6] {0,1};
    \node[down] (11) [above left= of 7] {1,0};
    \node (12) [above left= of 8] {1,0};
    \node[down] (13) [above left= of 9] {1,0};
    \node (14) [above left= of 10] {0,1};
    \node[down] (15) [above right= of 10] {0,1};
    \node (16) [below right= of 3] {0,1};
    \node (17) [below right= of 6] {0,1};
    \node (18) [below right= of 10] {0,1};
    \node (19) [below right= of 15] {0,1};
    \node (20) [above right= of 19] {0,1};

    \path (11) edge (7);
    \path (12) edge (8);
    \path (13) edge (9);
    \path (14) edge (10);
    \path (15) edge (19);
    \path (12) edge (7);
    \path (13) edge (8);
    \path (14) edge (9);
    \path (15) edge (10);
    \path (20) edge (19);
    \path (7) edge (4);
    \path (8) edge (5);
    \path (9) edge (6);
    \path (10) edge (18);
    \path (8) edge (4);
    \path (9) edge (5);
    \path (10) edge (6);
    \path (19) edge (18);
    \path (4) edge (2);
    \path (5) edge (3);
    \path (6) edge (17);
    \path (5) edge (2);
    \path (6) edge (3);
    \path (18) edge (17);
    \path (2) edge (1);
    \path (3) edge (16);
    \path (3) edge (1);
    \path (17) edge (16); 
    
    \end{tikzpicture}
\end{figure}

\begin{corollary}
    The spectrum of nim numbers of vertex geometry of a tree is $\{0,1,2,3\}$
\end{corollary}

\begin{theorem}
    The spectrum of nim numbers for the sum of games of vertex geometry of a tree is $\{0,1,2,3\}$.
\end{theorem}
\begin{proof}
    We use the bitwise XOR operator to compute the nim number of the sum of the game. The nim numbers $\{0,1,2,3\}$ are obtainable because $0 \oplus0 = 0$, $0 \oplus 1=1$, $0 \oplus 2=2$, and $0 \oplus 3 = 3$. The nim number is closed under the XOR operator because bitwise XOR operator only computes two bits at a time (since $3 = 11_2$), so the nim number of the sum has to be greater than or equal to $11_2 = 3$. Thus, we are done.
\end{proof}

We can also extend the concept of convex geometries to a tree graph $T$ with an edge set $S$. Similar to vertex geometries, edge sets of connected subgraphs of $T$ form a convex geometry on $S$. By finding the nim number, we completely solve the question proposed in \cite{McCoy} on edge geometries.

\begin{definition}
    Suppose that $W$ is the winning set of the avoidance game played on the edge geometry. We define $L(W)$ to be the set of connected components in $T\backslash W$.
\end{definition}

\begin{figure}[htp]\centering
    \begin{tikzpicture}[
        scale=4,line width=1pt,
        node distance=0.8cm and 1.2cm, 
        every node/.style={draw, circle, align=center, minimum size=6pt, inner sep=1.2pt},
        every edge quotes/.style = {auto, font=\footnotesize, sloped}
    ]

    \node (1) at (0,0) {};
    \node (2) [left =of 1] {};
    \node (3) [left =of 2] {};
    \node (4) [above left= of 3] {};
    \node (5) [below left= of 3] {};
    \node (6) [right =of 1] {};
    \node (7) [above right =of 6] {};
    \node (8) [above right =of 7] {};
    \node (9) [below right =of 7] {};
    \node (10) [below right =of 6] {};

    \path (1) edge ["A"] (2);
    \path (3) edge ["B"] (2);
    \path (4) edge ["C"] (3);
    \path (5) edge ["D"] (3);
    \path (1) edge ["E"] (6);
    \path (6) edge ["F"] (7);
    \path (7) edge ["G"] (8);
    \path (7) edge ["H"] (9);
    \path (10) edge ["I"] (6);

    \end{tikzpicture}
\end{figure}

\begin{figure}[htp]\centering
    \begin{tikzpicture}[
        scale=4,line width=1pt,
        node distance=0.8cm and 1.2cm, 
        every node/.style={draw, circle, align=center, minimum size=6pt, inner sep=1.2pt},
        every edge quotes/.style = {auto, font=\footnotesize, sloped}
    ]

    \node (1) at (0,0) {};
    \node (2) [left =of 1] {};
    \node (3) [left =of 2] {};
    \node (4) [above left= of 3] {};
    \node (5) [below left= of 3] {};
    \node (6) [right =of 1] {};
    \node (7) [above right =of 6] {};
    \node (8) [above right =of 7] {};
    \node (9) [below right =of 7] {};
    \node (10) [below right =of 6] {};

    \path (3) edge ["B"] (2);
    \path (4) edge ["C"] (3);
    \path (5) edge ["D"] (3);
    \path (1) edge ["E"] (6);
    \path (6) edge ["F"] (7);
    \path (7) edge ["G"] (8);
    \path (7) edge ["H"] (9);
    \path (10) edge ["I"] (6);

    \end{tikzpicture}
\end{figure}

\begin{figure}[htp]\centering
    \begin{tikzpicture}[
        scale=4,line width=1pt,
        node distance=0.8cm and 1.2cm, 
        every node/.style={draw, circle, align=center, minimum size=6pt, inner sep=1.2pt},
        every edge quotes/.style = {auto, font=\footnotesize, sloped}
    ]

    \node (1) at (0,0) {};
    \node (2) [left =of 1] {};
    \node (3) [left =of 2] {};
    \node (4) [above left= of 3] {};
    \node (5) [below left= of 3] {};
    \node (6) [right =of 1] {};
    \node (7) [above right =of 6] {};
    \node (8) [above right =of 7] {};
    \node (9) [below right =of 7] {};
    \node (10) [below right =of 6] {};

    \path (4) edge ["C"] (3);
    \path (5) edge ["D"] (3);
    \path (7) edge ["G"] (8);
    \path (7) edge ["H"] (9);
    \path (10) edge ["I"] (6);

    \end{tikzpicture}
\end{figure}
\begin{example}
    An example of a convex set for the edge geometry shown above are sets $\{A,B,C,D\}$ and $\{A,E,I,F\}$ as they are edges of a connected subgraph of the tree. However, $\{B,E,G\}$ and $\{A,F,G\}$ are not convex sets because they form a disconnected subgraph. In addition, $N(\{A\}) = \{\{B,C,D\}, \{E,F,G,H,I\}$ because these edge sets are the connected components when we remove $A$. Similarly, $N(\{B,A,E,F\}) = \{\{C,D\}, \{I\}, \{G,H\}\}$.
\end{example}

\begin{proposition}
    Let $\{S, \mathcal{K}\}$ be the edge geometry of a tree $T$ where $W \subseteq S$. Then, $w \in \text{Ex}(W)$ if and only if the elements of $W \backslash \{w\}$ are verteices in a single connected componenet of $T\backslash w$. 
\end{proposition}

\begin{proof}
    The proof is identical to the proof to Proposition 18 due to the similarity in their definition.
\end{proof}

There are two cases to consider to find the spectrum of nim numbers of edge geometries: $|W| = 1$ and $|W| \geq 2$. We first consider the case when $|W|=1$. Note by construction that $\mathcal{M} = N(W)$ by construction. $|\mathcal{M}| = |N(W)| \leq 2$ because $|W|=1$. The \emph{signature} of this game is $(e,o)$ where $e$ is the cardinality of $\{M \in \mathcal{M} \mid \text{pty}(M) = 0\}$ and $o$ is the cardinality of $\{M \in \mathcal{M} \mid \text{pty}(M) = 1\}$.

\begin{proposition}
    Let $(S, \mathcal{K})$ be the edge geometry of a tree $T$ and $W = \{w\} \subseteq S$. Given the signature of $(e,o)$, then 
    \[
    \text{nim}(\text{DNG}(S, W)) =  
    \begin{dcases} 
        0, & e=0, o=0\\
        0, & e =1, o=0 \\
        1, & e=0, o = 1 \\
        3, & e= 1, o=1
    \end{dcases}
    \]
\end{proposition}
\begin{proof}
    The proof is isomorphic to the proof of Proposition 21 because $\mathcal{M}$ of the edge geometry is essentially a limited version of the $\mathcal{M}$ in the vertex geometry.
\end{proof}

Now, we consider the case when $|W| \geq 2$. We use similar notation as vertex geometry. If $w \in \text{Ex}(W)$, then $T \backslash \{w\}$ has exactly one component that contains some element of $W$. Call this component $M_w$. $M_w$ component is a maximally non-generating set because the convex closure of its options immediately include $w$ which is therefore a superset of $W$. Notice from the construction that $\mathcal{M} = \{M_w \mid w \in \text{Ex}(W)\}$. We also use the notation $E_w := S \backslash M_w$ for simplicity.

Suppose there is $A \subseteq S$. We define 
\[
M_A := \bigcap\{M_w \mid w \in A\}
\]
The deficiency of a structure class \( M_A \) is given by \( \delta(M_A) := |A| \). Additionally, the signature of a structure class \( X_{M_A} \) is represented as \( \sigma(X_{M_A}) := (e, o) \), where \( e \) denotes the number of elements \( a \in A \) for which \( pty(V_a) = 0 \), and \( o \) represents the number of elements \( a \in A \) for which \( pty(V_a) = 1 \). The deficiency follows the relation \( \delta(M_A) = e + o \). The overall signature of the game is given by \( \sigma(X_{\emptyset}) \).

\begin{proposition}
    Let $(S, \mathcal{K})$ be the vertex geometry of a tree $T$ and $W \subseteq S$ such that $|W| \geq 2$. Given the signature of the game is $(e,o)$, 
    \[
    \text{nim}(\text{DNG}(S,W)) = \begin{dcases}
        1, & o>e \\
        0, & o<e\\
        3, & o=e, o \equiv 1 \pmod2\\
        2, & o=e, o \equiv 0\pmod2\\
    \end{dcases}
    \]
    when $\text{pty}(S) = 0$ and 
    \[
    \text{nim}(\text{DNG}(S,W)) = \begin{dcases}
        0, & o>e \\
        1, & o<e\\
        3, & o=e, o \equiv 1 \pmod2\\
        2, & o=e, o \equiv 0\pmod2\\
    \end{dcases}
    \]
    when $\text{pty}(S) = 1$.
\end{proposition}
\begin{proof}
    The proof is isomorphic to the proof for Proposition 31.
\end{proof}
\begin{corollary}
    The spectrum of nim numbers of edge geometry of a tree is $\{0,1,2,3\}$.
\end{corollary}
\begin{theorem}
    The spectrum of nim numbers for the sum of edge geometries is $\{0,1,2,3\}$. 
\end{theorem}
\begin{proof}
    The proof is isomorphic to the proof to Theorem 33 because the nim numbers in Corollary 32 and Corollary 39 are identical.
\end{proof}
\section{Conjectures and Future Research}
In this paper, we have explored the avoidance game within the framework of well-established convex geometries. However, the complete characterization of the spectrum of nim numbers for this game remains open. Considerable work remains in mapping the full theoretical landscape of this problem, particularly since convex geometries are not constrained by Lagrange’s Theorem, as seen in the group-theoretic version of the avoidance game as seen in \cite{Ernst}. 

Avoidance games on convex geometries, while highly theoretical, have broader applications in various fields of mathematics and computer science. In artificial intelligence and game strategy development, understanding the structure of impartial games such as $\text{DNG}$ provides insights into optimal decision-making and computational complexity. The study of nim numbers and structural equivalence in game theory can also be applied to algorithmic game playing, where AI systems analyze move sequences to determine winning strategies. Additionally, these concepts relate to optimization problems in discrete mathematics, where avoiding specific configurations—similar to avoidance games—can model real-world constraints, such as network security and resource allocation problems.

Furthermore, convex geometries and their associated game-theoretic properties are closely tied to combinatorial optimization and graph theory. Many real-world problems, such as sensor placement in networks, clustering in data science, and even social influence models, involve constraints that can be framed in terms of convex structures. By studying avoidance games in these settings and finding equivalence relations, we gain a deeper understanding of how structured constraints impact decision-making. Future research could explore whether the methods developed in this paper extend to more complex structures, such as higher-dimensional geometric spaces or dynamic graph models, further enhancing their relevance to applied mathematics and theoretical computer science.

In addition, we conjecture that the spectrum of the nim number of avoidance games played in convex geometries is all nonnegative integers $\{0,1,2,3,\dots\}$. One way that we could possibly go about proving this is finding a certain construction in $\mathbb{R}^2$ and extending the diagram to $\mathbb{R}^3$ and to higher dimensions so that the nim numbers increase.

\subsection*{Acknowledgements}

Thank you Dr. Thorne for mentoring me through this research.


\end{document}